\documentclass[a4paper,12pt]{amsart}
\usepackage{amsthm}
\usepackage{amsfonts}
\usepackage{amsmath}
\usepackage{amssymb}
\usepackage[pdftex]{graphicx}
\usepackage{hyperref}

\newtheorem{thm}{Theorem}[section]
\theoremstyle{definition}
\newtheorem{definition}[thm]{Definition}
\theoremstyle{remark}

\theoremstyle{plain}
\newtheorem{lem}[thm]{Lemma}
\newtheorem{Corollary}[thm]{Corollary}

\newtheorem{proposition}[thm]{Proposition}

\makeatletter
\@namedef{subjclassname@2020}{\textup{2020} Mathematics Subject Classification}
\makeatother

\newcommand{\bq}{/\!\!/}
\newcommand{\diag}{\operatorname{diag}}

\begin{document}
\title{Quasi-positive curvature on Bazaikin spaces}
\author{Jason DeVito and Evan Sherman}
\date{}


\maketitle

\begin{abstract}
We completely characterize the sectional curvature of all of the $13$-dimensional Bazaikin spaces.  In particular, we show that all Bazaikin spaces admit a quasi-positively curved Riemannian metric, and that, up to isometry, there is a unique Bazaikin space which is almost positively curved but not positively curved.
\end{abstract}

\section{Introduction}

Given a $5$-tuple of odd integers $\overline{q} = (q_1, ..., q_5)$ with $\gcd(q_1,...,q_5) = 1$, one can define an action of $Sp(2)\times S^1$ on $SU(5)$ as follows.  Given $(A+Bj)\in Sp(2)$ with $A$ and $B$ complex $2\times 2$ matrices, $z\in S^1$, and $C\in SU(5)$, we have $$(A+Bj, z)\ast C = \diag(z^{q_1},...,z^{q_5})\, C\, \begin{bmatrix} A & B & 0\\ -\overline{B} & \overline{A} & 0 \\ 0 & 0 & z^q \end{bmatrix}^{-1},$$ where $q = \sum q_i$.

As is well known, the above action is effectively free if and only if $$\gcd(q_{\sigma(1)} + q_{\sigma(2)}, q_{\sigma(3)} + q_{\sigma(4)}) = 2$$ for every $\sigma\in S_5$.  If $\overline{q} = (q_1,..., q_5)$ verifies this condition, the orbit space, denoted $\mathcal{B}_{\overline{q}}$ is a smooth manifold called a Bazaikin space.

Bazaikin spaces were introduced by Bazaikin \cite{Baz1}, where he showed an infinite family of them admit Riemannian metrics of positive sectional curvature. Specifically, beginning with a bi-invariant Riemannian metric $\langle \cdot ,\cdot \rangle_0$ on $SU(5)$, he Cheeger deformed it in the direction of $U(4)\subseteq SU(5)$ obtaining a left $SU(5)$-invariant, right $U(4)$-invariant metric $\langle \cdot ,\cdot \rangle_1$.  The above action is isometric with respect to $\langle \cdot , \cdot \rangle_1$, so there is an induced metric $\langle \cdot, \cdot \rangle$ on $\mathcal{B}_{\overline{q}}$.  We will refer to this metric as \textit{the natural metric} on $\mathcal{B}_{\overline{q}}$.

With respect to the natural metric, much is already known.  First, since bi-invariant metrics have non-negative sectional curvature, and curvature is non-decreasing along Riemannian submersions \cite{On1}, it follows that every $\mathcal{B}_{\overline{q}}$ has non-negative sectional curvature.  In addition, $\mathcal{B}_{\overline{q}}$ has positive sectional curvature at every point if and only if all sums $q_i + q_j$ for every distinct $i,j$ are positive, or they are all negative \cite{Baz1,DE}.  Apart from the $7$-dimensional Eschenburg spaces \cite{AW,Es2}, these form the only known infinite family of positively curved spaces in a fixed dimension.

In addition, Kerin \cite{Ke1} has shown that if four of the $q_i$ have the same sign, then $\mathcal{B}_{\overline{q}}$ is quasi-positively curved and that $\mathcal{B}_{\overline{q}}$ for $\overline{q} = (1,1,1,1,-1)$ is almost positively curved.  Recall that a Riemannian manifold is called \textit{quasi-positively curved} if it has non-negative sectional curvature and it has a point for which the sectional curvature of all two-planes at that point is positive.  A Riemannian manifold is called \textit{almost positively curved} if the set of points with all two-planes positively curved is open and dense.

Our main result characterizes the sectional curvature of all Bazaikin spaces with respect to the natural metric.

\begin{thm}\label{thm:main}Suppose $\mathcal{B}_{\overline{q}}$ is a Bazaikin space with metric $\langle \cdot ,\cdot \rangle$ as described above.  Then

\begin{itemize}\item  $\mathcal{B}_{\overline{q}}$ is quasi-positively curved if and only if $\overline{q}$ is not a permutation of $\pm(1,1,1,-1,-3)$.

\item  $\mathcal{B}_{\overline{q}}$ is almost positively curved if and only if it is strictly positively curved or $\overline{q}$ is a permutation of $\pm(1,1,1,1,-1)$.

\end{itemize}

\end{thm}

As shown in \cite{EKS}, the diffeomorphism type of $\mathcal{B}_{\overline{q}}$ is unchanged by replacing any one $q_i$ with $-\sum_{\ell = 1}^5 q_\ell$.  Thus, we observe that the unique Bazaikin space whose natural metric has zero-curvature planes everywhere is diffeomorphic to the unique Bazaikin space whose natural metric is almost positively curved but not positively curved.

\begin{Corollary}  Up to diffeomorphism, every Bazaikin space admits a metric of quasi-positive curvature.
\end{Corollary}

The outline of this paper is as follows.  Section \ref{sec:qp} contains more detailed background on Bazaikin spaces, their geometry, and their topology.  Section \ref{sec:open} begins with a proof that, apart from $\overline{q} = (1,1,1,-1,-3)$, every Bazaikin space is quasi-positively curved.  The remainder of Section \ref{sec:open} is devoted to the construction of an open subset of $\mathcal{B}_{\overline{q}}$ with zero-curvature planes under the assumption that the quantities $q_i + q_j$ can be both positive and negative.  Having accomplished this, the proof of Theorem \ref{thm:main} is an easy consequence.

\textbf{Acknowledgements}:  The first author was supported by NSF Grant DMS-2105556.  He is grateful for their support.  Both authors would like to thank Martin Kerin for helpful conversations.

\section{Background}\label{sec:qp}

\subsection{Cheeger deformations}  We begin with the background information on Cheeger deformations of biquotients.  Biquotients are defined as follows:  given a compact Lie group $G$, any closed subgroup $H\subseteq G\times G$ naturally acts on $G$ via $(h_1,h_2)\ast g = h_1 g h_2^{-1}$.  When the action is effectively free, the orbit space is denoted by $G\bq H$ and is called a biquotient.  For more detailed background on biquotients, see \cite{Es2}.

Cheeger deformations, introduced by Cheeger \cite{Ch1}, are metric deformations which tend to increase curvature at the expense of symmetry.

In more detail, suppose $G$ is a compact Lie group and $K\subseteq G$ is a closed subgroup.  If $\langle\cdot,\cdot\rangle_0$ is a fixed bi-invariant metric on $G$, then one can consider the Riemannian metric $\langle \cdot, \cdot\rangle_0 + t \langle \cdot,\cdot\rangle_0|_K$ on $G\times K$ where $t\in (0,1)$ is some fixed parameter.  Then $K$ acts isometrically on $G\times K$ via $k\ast(g,k') = (gk^{-1}, kk')$ and hence the quotient $(G\times K)/K$ inherits a Riemannian metric.  The $G\times K$ action on $(G\times K)/K$ given by $(g,k)\ast [(g',k')] = [(gg', k' k^{-1})]$ is isometric.  The map $(g,k')\mapsto gk'\in G$ descends to a diffeomorphism $(G\times K)/K\rightarrow G$ and under this diffeomorphism the induced $G\times K$ action on $G$ is simply by left and right multiplication.  Hence, by declaring this diffeomorphism to be an isometry, we obtain a metric $\langle\cdot,\cdot\rangle_1$ on $G$ which is left $G$-invariant and right $K$-invariant, the so-called \textit{Cheeger deformation} of $\langle \cdot, \cdot\rangle_0$ in the direction of $K$.  From O'Neill's formula for a Riemannian submersion \cite{On1} together with the well-known fact that $\langle \cdot, \cdot \rangle_0$ is non-negatively curved, it follows that $\langle \cdot, \cdot \rangle_1$ is also non-negatively curved.

If $H\subseteq G\times K$, then $H$ acts on $(G,\langle\cdot,\cdot\rangle_1)$ isometrically via $(h_1,h_2)\ast g = h_1 g h_2^{-1}$.  When this action is effectively free, the biquotient $G\bq H$, inherits a Riemannian metric $\langle \cdot, \cdot\rangle$ from $\langle \cdot, \cdot \rangle_1$.  Again, via O'Neill's formula \cite{On1}, we see that $G\bq H$ is non-negatively curved.  Further, any zero-curvature plane $\sigma \subseteq T_{[g]} G\bq H$ must lift to horizontal zero-curvature plane in $(G\times K, \langle \cdot, \cdot \rangle_0 + t\langle \cdot, \cdot \rangle_0)$, and zero-curvature planes in this latter space are well understood.  We also note that while, in principle, O'Neill's formula allows a horizontal zero-curvature plane in $G\times K$ to project to a positive curvature plane in $G\bq H$, Wilking and Tapp \cite{Wi,Ta2} have shown that in the special case considered above, this does not occur. 

\

For the duration of the paper, we set $$G= SU(5), K = U(4), H = Sp(2)\times S^1$$ where $K$ is embedded via $$k\mapsto \diag(k,\overline{\det k}),$$ and where $H$ is embedded into $G\times K\subseteq G\times G$ via $$(A+Bj, z)\mapsto \left(\diag(z^{q_1},..., z^{q_5}), \begin{bmatrix} A & B & 0\\ -\overline{B} & \overline{A} & 0 \\ 0 &0 & z^q\end{bmatrix}\right).$$  Here, each $q_i\in \mathbb{Z}$, $\gcd(q_1,...,q_k) = 1$, and $q = \sum q_i$.

Focusing on the embedding $Sp(2)\subseteq SU(5)$ given by $(A+Bj)\mapsto \begin{bmatrix} A& B & 0 \\ -\overline{B} & \overline{A} & 0 \\ 0 & 0 & 1\end{bmatrix}$ above, we observe the following two facts:

\begin{proposition}\label{prop:emb}

For any $h = (h_{ij})\in Sp(2)\subseteq SU(5)$ with $1\leq i,j\leq 4$, we have $|h_{i2}| = |h_{j 4}|$ whenever $|i-j| = 2$.  In addition, given any vector $v\in S^7\subseteq \mathbb{C}^4$, there is an element $h\in Sp(2)\subseteq SU(5)$ for which $v$ comprises the first four entries of the second column of $h$.

\end{proposition}
\begin{proof}The first statement is obvious because the entries of $A$ and $\overline{A}$ have the same length, as do the entries of $B$ and $-\overline{B}$.  The second statement is an immediate consequence of the fact that the standard representation of $Sp(2)$ on $\mathbb{H}^2\cong \mathbb{C}^4$ acts transitively on the unit sphere.

\end{proof}

We use the bi-invariant metric whose value at identity $I\in G$ is $\langle X,Y\rangle_0 = -ReTr(XY)$ with $X,Y\in T_{I} G$.  We recall that the induced metric on $\mathcal{B}_{\overline{q}}$ is referred to as \textit{the natural metric}.

We note that, strictly speaking, $H$ is not a subgroup of $SU(5)\times SU(5)$, but rather, is a subgroup of $U(5)\times N_{U(5)}(K)$ where $N_{U(5)}(K)$ denotes the normalizer of $K$ in $U(5)$.  However, the biquotient action of $H$ on $U(5)$ preserves $SU(5)$ and acts isometrically.  Hence, all of the above still applies in this slightly more general case.

As is well known, this action is free iff all $q_i$ are odd and $$\gcd(q_{\sigma(1)} + q_{\sigma(2)}, q_{\sigma(3)} + q_{\sigma(4)}) = 2$$ for all permutations $\sigma\in S_5$.

It is clear that replacing $\overline{q}$ with $-\overline{q}$ determines an  equivalent action.  We therefore can and will always assume at least three $q_i$ are positive, motivating the following definition.

\begin{definition}An element $\overline{q} = (q_1,...,q_5)\in \mathbb{Z}^5$ is called  \textit{admissible} if \begin{itemize} \item All $q_i$ are odd,

\item  $\gcd(q_{\sigma(1)} + q_{\sigma(2)}, q_{\sigma(3)} + q_{\sigma(4)}) = 2$ for all $\sigma \in S_5$, and 

\item  at least three $q_i$ are positive.

\end{itemize}

\end{definition}

Given an admissible $\overline{q}$, it is well known that one can modify the $q_i$ to obtain diffeomorphic biquotients.  We summarize these allowable modifications in the following proposition.

\begin{proposition}\label{prop:isom}  Suppose $\overline{q}$ is admissible.

\begin{itemize} \item  The biquotients $\mathcal{B}_{\overline{q}}$ and $\mathcal{B}_{-\overline{q}}$ are isometric when equipped with their natural metrics.

\item  If $\overline{r}$ is any permutation of $\overline{q}$, then the biquotients $\mathcal{B}_{\overline{q}}$ and $\mathcal{B}_{\overline{r}}$ are isometric when equipped with their natural metrics.

\item \cite{EKS}  If $\overline{r}$ is obtained from $\overline{q}$ by replacing one $q_i$ with $-\sum_{\ell=1}^5 q_\ell$, then $\mathcal{B}_{\overline{q}}$ and $\mathcal{B}_{\overline{r}}$ are diffeomorphic, but, in general, are non-isometric when equipped with their natural metrics.

\end{itemize}

\end{proposition}

As mentioned previously, Bazaikin spaces have been extensively studied and, as such, conditions controlling curvature are well understood.  For example, based on work by Dearicott and Eschenburg \cite{DE}, Kerin \cite[Lemma 3.1, Lemma 3.2]{Ke1} proves the following.

\begin{proposition}\label{prop:curvcond}  Suppose $\overline{q} = (q_1,...,q_5)$ is admissible.  For $A =(A_{ij})\in SU(5)$, there is a zero-curvature plane at the point $[A]\in \mathcal{B}_{\overline{q}}$ if and only if at least one of the following two conditions is satisfied:

\begin{align}
     \sum_{\ell=1}^5 q_{\ell} &= \sum_{\ell=1}^{5}|A_{\ell5}|^2 q_\ell\label{eq1}\\
    0 &= \sum_{\ell=1}^{5}(|(Ah)_{\ell 2}|^2+|(Ah)_{\ell4}|^2)q_\ell\label{eq2}
\end{align}
for some $h \in Sp(2) \subseteq SU(5).$
\end{proposition}

This proposition will be key to proving Theorem \ref{thm:main}.  However, before proving Theorem \ref{thm:main}, we take a brief digression on the topology of the new examples.

\subsection{The topology of Bazaikin spaces}  In \cite{FloritZiller}, Florit and Ziller compute the topology of $\mathcal{B}_{\overline{q}}$.  In order to describe their results, it is convenient to define $q_6:=-(q_1+q_2+q_3 + q_4 + q_5)$.  Note that, from Proposition \ref{prop:isom}, the Bazaikin spaces corresponding to a $5$-tuple made by deleting one entry of $(q_1,...,q_6)$ are all diffeomorphic.  Using the shorthand $\sigma_i$ to denote $\sigma_i(q_1,q_2,q_3,q_4,q_5, q_6)$, the elementary symmetric polynomial of degree $i$ in the variables $q_i$, they prove:

\begin{thm}[Florit-Ziller]\label{thm:top}  For a Bazaikin space $\mathcal{B}_{\overline{q}}$ the cohomology ring is determined up to isomorphism by $H^8(\mathcal{B}_{\overline{q}})\cong \mathbb{Z}_s$ with $s = \sigma_3/8$.  Further, the Pontryagin classes are given by $$p_1 = -\sigma_2 u^2 \in H^4(\mathcal{B}_{\overline{q}};\mathbb{Z})\cong \mathbb{Z} \text{ and } p_2 = \frac{3p_1^2 - \sigma_4}{8} u^4 \in H^8(\mathcal{B}_{\overline{q}};\mathbb{Z}),$$ where $u\in H^2(\mathcal{B}_{\overline{q}})\cong \mathbb{Z}$ is a generator. 

\end{thm}

Florit and Ziller prove that $\sigma_3/8 \cong \pm 1\pmod{6}$, so $8\in \mathbb{Z}_s$ has a multiplicative inverse.  They also note that $-\sigma_2 = \frac{1}{2} \sum_{i=1}^6 q_i^2$, so a given value of $p_1$ can only occur for finitely many Bazaikin spaces.

Below, we shall prove that all Bazaikin spaces are quasi-positively curved.  Previously, this was only known for the subset for which at least four of $(q_1,...,q_5)$ are positive.  We now show that new topological types appear among this enlarged class of quasi-positively curved Bazaikin spaces.

\begin{proposition}\label{prop:newex}The Bazaikin space $\mathcal{B}_{\overline{q}}$ with  $\overline{q} = (7,1,1,-3,-3)$ is not homeomorphic to any Bazaikin space $\mathcal{B}_{\overline{r}}$ with four entries of $\overline{r} = (r_1,...,r_5)$ positive.

\end{proposition}

\begin{proof}  Suppose $\overline{r} = (r_1,...,r_5)$ is admissible and set $r_6:=-\sum_{\ell = 1}^5 r_\ell$.  Assume that  $\mathcal{B}_{\overline{q}}$ is homeomorphic to $\mathcal{B}_{\overline{r}}$.  We will first show that $\overline{r}$ must be related to $\overline{q}$ via the operations in Proposition \ref{prop:isom}.

Using Theorem \ref{thm:top}, one easily computes that $p_1(\mathcal{B}_{\overline{q}}) = 39$, where we have chosen an identification of $H^4(\mathcal{B}_{\overline{q}})$ with $\mathbb{Z}$.  Novikov \cite{No} has shown that the rational Pontryagin classes are homeomorphism invariants, so we must have $p_1(\mathcal{B}_{\overline{r}}) = 39$ as well.  We will show that there is precisely one other diffeomorphism type of Bazaikin space with $p_1 = 39$, corresponding to the integers $(5,5,-3,-3,-3,-1)$ and that the order of $H^8$ distinguishes them.

Since $p_1(\mathcal{B}_{\overline{r}}) = 39$, we find \begin{equation}\label{eqn:sos}78 = \sum_{\ell = 1}^6 r_i^2.\end{equation}

Obviously each $r_i$ is bounded above by $\lfloor\sqrt{78}\rfloor = 8$.  Since each $r_i$ is odd, we have $|r_i| \leq 7$.  In addition, we must obviously have at least one $|r_i| \geq 5$.

Thus, by reordering the $r_i$ and replacing $\overline{r}$ by $-\overline{r}$ if necessary, we may assume $r_1 \in \{ 5,7\}$ has the largest absolute value of all the $r_i$.  In addition, we assume that $|r_{i+1}|\leq |r_i|$ for all $i$.

Assume initially $r_1 = 7$.  Then, it follows from Equation \eqref{eqn:sos} that $|r_2|\in \{3,5\}$.  If $|r_2| = 5$, it is easy to see $(|r_i|) = (7,5,1,1,1,1)$.  The only choice of signs leading which satisfies $\sum r_i = 0$ is $(7,-5,1,1,-1,1)$.  But then $gcd(r_4 + r_5, r_1 + r_3) = 8 > 2$, so this is not admissible.  If $|r_2| = 3$, then it is easy to see that $(|r_i|) = (7,3,3,3,1,1)$.  Here, the only choice of signs which satisfies $\sum r_i = 0$ is $(7,-3,-3,-3,1,1)$.  Deleting a $-3$ gives $\mathcal{B}_{\overline{q}}$ above.  This concludes the case where $r_1 = 7$.

So, next assume $r_1 = 5$.  Since $5\cdot 3^2 \leq 78-25$, and $|r_2|\leq r_1$, it easily follows that $|r_2| =5$.  From here, it is easy to see that there are precisely two possibilities for $(|r_i|):  (5,5,5,1,1,1)$ and $(5,5,3,3,3,1)$.  In the first case, there is no way to assign signs to make the sum zero.  In the second case, the unique way to do so is as $(5,5,-3,-3,-3,-1)$.  One can check that $-(5,5,-3,-3,-3,-1)$ is, in fact, admissible.

Using Theorem \ref{thm:top} again, one easily computes that $$|\sigma_3(7,-3,-3,-3,1,1)| = 88 \neq 56 = |\sigma_3(5,5,-3,-3,-3,-1)|,$$ so these two examples cannot be homotopy equivalent.

From Proposition \ref{prop:isom}, we therefore conclude that if $\mathcal{B}_{\overline{q}}$ is homeomorphic to $\mathcal{B}_{\overline{r}}$, then $(r_1,...,r_6)$ is, up to sign, a permutation of $(7,-3,-3,-3,1,1)$.   Since both $\pm (7,-3,-3,-3,1,1)$ consists of three positive and three negative entries, no such choice of $\overline{r}$ can have $4$ $r_i$ of the same sign.
\end{proof}

Apart from Bazaikin spaces, the only known $13$-dimensional simply connected manifolds admitting quasi-positively curved metrics are diffeomorphic to $S^{13}$, a circle quotient of $S^7\times S^7$ \cite{Ke2}, and $T^1 S^6$, the unit tangent bundle of $S^6$ \cite{Wi}.  These spaces have the rational cohomology groups of $S^{13}$, $\mathbb{C}P^3\times S^7$, and $S^{13}$ respectively.  Since every Bazaikin space has the rational cohomology groups of $\mathbb{C}P^2\times S^9$ every Bazaikin space is homotopically distinct from these three examples.  Thus, the space $\mathcal{B}_{\overline{q}}$ of Proposition \ref{prop:newex} is not homeomorphic to any previously known manifold admitting a metric of quasi-positive curvature.

Lastly, we remark that we have written a computer program which finds all Bazaikin spaces with a given $p_1\leq 9615$.  Among these, there are over $87000$ Bazaikin spaces, including almost $52000$ which were not previously known to admit a metric of quasi-positive curvature.  Using Theorem \ref{thm:top}, we found none of the new examples are homeomorphic to any previously known example, except when related via Proposition \ref{prop:isom}.  This suggests that there are infinitely many homeomorphism types among our new examples, but we were unable to prove it.  See also \cite{FloritZiller} for similar computer calculations among an even larger range.

\section{Metric properties of Bazaikin spaces}\label{sec:open}

In this section, we begin by showing that, apart from permutations of $\overline{q} = (1,1,1,-1,-3)$, the natural metric on all Bazaikin spaces is quasi-positively curved.  We then show that this exceptional Bazaikin space, in fact, has zero-curvature planes at every point.  The second subsection contains the proof that apart from the previously known examples admitting strictly positive or almost positive sectional curvature, all the remaining examples have open sets of points containing at least one zero-curvature plane.

\subsection{Quasi-positive curvature}

The goal of this subsection is to show that the natural metric on $\mathcal{B}_{\overline{q}}$ is quasi-positively curved, except when $\overline{q}$ is a permutation of $(1,1,1,-1,-3)$.  We begin with the key property which sets $(1,1,1,-1,-3)$ apart from all other admissible $5$-tuples.

\begin{lem}\label{lemma:dblsum}
Suppose $\overline{q}=(q_1,...,q_5)$ is admissible and not a permutation of $(1,1,1,-1,-3)$.  Then there are distinct $a,b,c,d\in \{1,2,3,4,5\}$ such that $q_a + q_b$ and $q_c + q_d$ are either both positive or both negative.

\end{lem}
\begin{proof}Recall from Proposition \ref{prop:isom} that permutations of the $q_i$ give rise to isometric Bazaikin spaces.

If at least four of $\overline{q}=(q_1,...,q_5)$ are positive, say $q_1,q_2,q_3,q_4>0$. Then one can take $(a,b,c,d) = (1,2,3,4)$.

Thus, we may assume $q_1,q_2,$ and $q_3$ are positive and $q_4$ and $q_5$ are negative.  If at least one of  $|q_4|,|q_5|$ is less than at least one of  $q_1,q_2,q_3$, say $|q_4|<q_1$ without loss of generality, then one can take $(a,b,c,d) = (1,4,2,3)$. 

Thus, we may assume both $|q_4|,|q_5|$ are greater than or equal to all $q_1,q_2,q_3$.  If both $|q_4|,|q_5|$ greater than each of $q_1,q_2,q_3$, then one may take $(a,b,c,d) = (1,4,2,5)$.

So, we may assume that equality occurs; say $|q_4|=q_1$ without loss of generality. Since $q_1+q_4=0$, $2 = \gcd(q_2+q_3,q_1 + q_4)$ implies $q_2+q_3=2$, that is, that $q_2=q_3=1$.  Likewise, $2 = \gcd(q_2+q_5, q_1+q_4)$, which then implies $q_5=-3$.  Recalling that $|q_5|\geq |q_1|$, this implies $q_1 \in \{1,3\}$, giving rise to $(1,1,1,-1,-3)$ and $(3,1,1,-3,-3)$.  In the latter case, one can take $(a,b,c,d) = (2,4,3,5)$.

\end{proof}

\begin{thm}\label{thm:qp}
If $\overline{q} = (q_1,...,q_5)$ is admissible and not a permutation of $(1,1,1,-1,-3)$, then the natural metric on $\mathcal{B}_{\overline{q}}$ is quasi-positively curved.
\end{thm}

\begin{proof}
 Suppose $\overline{q} = (q_1,...,q_5)$ is admissible and that $\overline{q}$ is not a permutation of $(1,1,1,-1,-3)$. Then by Lemma~\ref{lemma:dblsum}, there is a four element subset $\{a,b,c,d\}\subseteq \{1,2,3,4,5\}$ for which $q_a + q_b$ and $q_c + q_d$ have the same sign.  Since permuting the $q_i$ gives isometric Bazaikin spaces, we may reorder $\overline{q} = (q_1,...,q_5)$ so that $q_1+q_3$ and $q_2+q_4$ have the same sign.

By Proposition~\ref{prop:curvcond}, we need only find $A\in SU(5)$ for which neither Equation~\eqref{eq1} nor Equation~\eqref{eq2} is satisfied.  In fact, we can take $A=I$, the identity matrix. Then, Equation~\eqref{eq1} becomes 
\begin{equation*}
    q_1+q_2+q_3+q_4+q_5=q_5,
\end{equation*}
which is not satisfied since $q_1+q_3$ and $q_2+q_4$ have the same sign.

Now, suppose $h\in Sp(2)\subseteq SU(5)$ is arbitrary.  As $Ah = h$, Equation \eqref{eq2} becomes $$ 0 = \sum_{\ell = 1}^5 (|h_{\ell 2}|^2 + |h_{\ell 4}|^2) q_\ell.$$  From Proposition \ref{prop:emb}, we know $|h_{12}|^2=|h_{34}|^2$, $|h_{14}|^2=|h_{32}|^2$, $|h_{22}|^2=|h_{44}|^2$, ${|h_{24}|^2=|h_{42}|^2}$, and $|h_{52}| = |h_{54}| = 0$. Therefore, Equation~\eqref{eq2} becomes \begin{equation*}
   0=(|h_{12}|^2+|h_{14}|^2)(q_1+q_3)+(|h_{22}|^2+|h_{24}|^2)(q_2+q_4).
  \end{equation*}
  Since $q_1+q_3$ and $q_2+q_4$ have the same sign, Equation~\ref{eq2} is not satisfied unless $h_{12},h_{14},h_{22},h_{24}=0$, which implies $h \notin Sp(2) \subseteq SU(5)$.  Thus $\mathcal{B}_{\overline{q}}$ is quasi-positively curved.
  \end{proof}

  We now show the hypothesis that $\overline{q}$ is not a permutation of $(1,1,1,-1,-3)$ is essential by finding zero-curvature planes at every point of this exceptional Bazaikin space.

\begin{proposition}\label{prop:exzero}  The natural metric on $\mathcal{B}_{\overline{q}}$ for $\overline{q} = ( 1,1,1,-1,-3)$ has a zero-curvature plane at every point.

\end{proposition}

\begin{proof}  From Proposition \ref{prop:curvcond}, it is sufficient to show that for every $A\in SU(5)$, there is an $h\in Sp(2)$ for which $$0 = \sum_{\ell = 1}^5 (|(Ah)_{\ell 2}|^2 + |(Ah)_{\ell 4}|^2) q_\ell.$$

Given $A\in SU(5)$, we let $g_A:Sp(2)\rightarrow \mathbb{R}$ be defined by $$g_A(h) = \sum_{\ell = 1}^5 (|(Ah)_{\ell 2}|^2 + |(Ah)_{\ell 4}|^2) q_\ell.$$  Because $Sp(2)$ is connected, in order to prove this proposition, it is sufficient to show that for each $A$, $g_A$ attains both non-positive and non-negative values.

Since $Ah\in SU(5)$ implies $\sum_{\ell= 1}^5 |(Ah)_{\ell i}|^2 = 1$ for each fixed $i$, we may rewrite $g_A$ as $$g_A(h) = 2 -  2(|(Ah)_{4 2}|^2 + |(Ah)_{4 4}|^2) - 4(|(Ah)_{5 2}|^2 + |(Ah)_{5 4}|^2)$$.

Given $A\in SU(5)$, we let $s_i = (A_{i1}, A_{i2}, A_{i3}, A_{i4})$ for $i = 4,5$.  That is, $s_4$ and $s_5$ consists of the first four entries of the fourth and fifth rows of $A$.  Given $h\in Sp(2)\subseteq SU(5)$, we let $h_2,h_4\in \mathbb{C}^4$ denote the first four entries of the $2$nd and $4$th columns of $h$, respectively.  Note that from the form of the embedding of $Sp(2)$ in $SU(5)$, $h_2$ completely determines $h_4$.  This notation allows us to express $g_A(h)$ as $$g_A(h) = 2 -  2(|(s_4\cdot h_2)|^2 + |(s_4\cdot h_4)|^2) - 4(|(s_5\cdot h_2)|^2 + |(s_5 \cdot h_4)|^2)$$ with $v\cdot w = \sum_i v_i w_i$ for $v,w\in \mathbb{C}^4$.

We begin by finding an $h$ for which $g_A(h)\geq 0$.  To that end, consider the system of equations $$\begin{cases}  s_5 \cdot h_2 = 0 \\ s_5\cdot h_4 = 0\end{cases} .$$  By breaking into real and imaginary parts, we view this as a homogeneous linear system of $4$ equations in $8$ unknowns (the real and imaginary parts of each entry of $h_2$) whose coefficients are the real and imaginary parts of the entries of $s_5$.  As $8>4$, we can always find a non-zero solution to this.  Letting $v$ denote any non-zero solution, scaled to have unit length, Proposition \ref{prop:emb} gives the existence of an $h\in Sp(2)$ with $h_2 = v$.  Thus, $s_5 \cdot h_2 = s_5\cdot h_4 = 0$.  It follows that $$g_A(h) =2 - 2(|(Ah)_{42}|^2 + |(Ah)_{44}|^2) \geq 0,$$ where the inequality follows because $Ah\in SU(5)$ implies the term in parenthesis is at most $1$.

\bigskip

We next turn attention to finding an $h$ for which $g_A(h) \leq 0$.  Assume initially that $|s_5|^2 \geq 1/2$.  Then, choosing $h$ with $h_2 = \frac{\overline{s}_5}{|s_5|}$, which is possible via Proposition \ref{prop:emb}, we find that $$g_A(h)\leq 2 - 4( |s_5\cdot h_2|^2) \leq 2- 4(1/2) \leq 0 ,$$ as desired.

Thus, we may assume $|s_5|^2 \leq \frac{1}{2}$.  Since $A\in SU(5)$, this implies $s_4\neq 0$.  Thus, we may find an $h\in Sp(2)$ with $h_2 = \frac{\overline{s}_4}{|s_4|}$.

Then $$g_A(h)\leq 2 - 2 \left|s_4\cdot \frac{\overline{s}_4}{|s_4|}\right|^2 - 4 \left| s_5\cdot \frac{\overline{s}_4}{|s_4|}\right|^2 = 2-2\frac{|s_4|^4 + 2|s_5\cdot\overline{s}_4|^2}{|s_4|^2}.$$

Note that since $A\in SU(5)$, $$|s_5\cdot \overline{s}_4|^2 = |A_{55} |^2 |\overline{A_{45}}|^2 = (1-|s_5|^2)(1-|s_4|^2).$$  Because $|s_5|^2 \leq \frac{1}{2}$, $$2|s_5\cdot \overline{s}_4|^2 = 2(1-|s_5|^2)(1-|s_4|^2) \geq 1-|s_4|^2,$$ and thus, \begin{equation}\label{eq:ineq}g_A(h)\leq 2 - 2\frac{|s_4|^4 + 1-|s_4|^2}{|s_4|^2}.\end{equation}  Since $$0\leq (|s_4|^2 - 1)^2 = |s_4|^4 -2|s_4|^2 + 1,$$ it follows that $$|s_4|^2\leq |s_4|^4 -|s_4|^2  +1,$$ and hence, that $$1\leq \frac{|s_4|^4 + 1-|s_4|^2}{|s_4|^2}.$$  Substituting this into Equation \eqref{eq:ineq}, we see that $$g_A(h) \leq 2 - 2\frac{|s_4|^4 + 1-|s_4|^2}{|s_4|^2} \leq 2-2\leq 0,$$ as desired.
\end{proof}

\subsection{Open sets with zero-curvature planes}

The main goal of this section is to prove the following theorem:

\begin{thm}\label{thm:notalmost}  Suppose $\overline{q}=(q_1,...,q_5)\in \mathbb{Z}^5$ is an admissible $5$-tuple of integers.  If the set $\{q_i+q_j\}$ contains both positive and negative integers, then $\mathcal{B}_{\overline{q}}$ is not almost positively curved.

\end{thm}

By permuting the $q_i$, we can and will assume that $q_1 + q_5 < 0$ and that $q_5<0$ while $q_2,q_3,q_4 > 0$. We let $q_m$ denote $\max\{q_2,q_3,q_4\}$.

Let $$A_0 = \begin{bmatrix} \cos \theta & 0 & 0 & 0 & \sin\theta\\ 0 & 1 & 0 & 0 & 0\\ 0 & 0 & 0 & 1 & 0\\ -\sin\theta &0 & 0 & 0 & \cos\theta\\ 0 & 0 & 1 & 0 & 0\end{bmatrix}\in SU(5)$$ where $0<\theta< \pi/2$ is a fixed real number small enough that $\sin\theta< \frac{1}{16\sqrt{q_m}}$.  We will find a neighborhood $A_0\in V\subseteq SU(5)$ for which every $B\in V$ contains a horizontal zero-curvature plane.  Assuming temporarily that we can accomplish this, the projection $\pi:SU(5)\rightarrow \mathcal{B}_{\overline{q}}$ maps $V$ to $\pi(V)\subseteq \mathcal{B}_{\overline{q}}$, giving an open subset witnessing the fact that $\mathcal{B}_{\overline{q}}$ is not almost positively curved.

As in the proof of Proposition \ref{prop:exzero}, given $B = (B)_{ij}\in SU(5)$, we let $s_i = s_i(B) = (B_{i1}, B_{i2}, B_{i3}, B_{i4})$ denote the first four entries in the $i$-th row of $B$.  We let $V\subseteq SU(5)$ consists of all $B\in SU(5)$ satisfying the following conditions:

\begin{enumerate} \item  $|B_{ij}| <   \frac{1}{16\sqrt{q_m}}$ for $(i,j)\not\in \{(1,1), (2,2), (3,4), (4,5), (5,3)\}$

\item $|s_5|^2 > 7/8$

\item  $|s_1| < |s_5|$

\item  $\sum_{\ell = 1}^5 \left( |B_{\ell 2}|^2 + |B_{\ell 4}|^2 \right)q_\ell > 0$ \end{enumerate}

\begin{proposition} The matrix $A_0\in V$, so $V\neq \emptyset$
\end{proposition}

\begin{proof}We verify the four conditions in order.  If $(i,j)\neq (1,5), (4,1)$, then the first condition becomes $|0| < \frac{1}{16\sqrt{q_m}}$, which is obviously true.  In the two exceptional cases, $|A_{ij}| = |\sin\theta| < \frac{1}{16\sqrt{q_m}}$ by the choice of $\theta$.

The second condition becomes $1 > 7/8$, which is obviously true.  The third condition becomes $|\cos\theta| < 1$, which is true since $0<\theta <\pi/2$.

The last condition becomes $q_2 + q_4 > 0$, which is true since $q_2$ and $q_4$ are both larger than zero by assumption.

\end{proof}

Now, given $B\in V$ and viewing $Sp(2)\subseteq SU(5)$, we define a function $f_B:Sp(2)\rightarrow \mathbb{R}$ by $$f_B(h) = \sum_{\ell = 1}^5 \left( |(Bh)_{\ell 2}|^2 + |(Bh)_{\ell 4}|^2 \right)q_\ell.$$  From Proposition \ref{prop:curvcond}, if $f_B$ has a zero for every $B\in V$, we will have established that every point in $V$ has at least one zero-curvature plane.  As in the proof of Proposition \ref{prop:exzero}, it is sufficient to show that $f_B$ achieves both positive and negative values.  The last condition defining $V$ asserts that $f_B(I) > 0$, where $I\in Sp(2)\subseteq SU(5)$ is the identity matrix.  Thus, we need only find $h\in Sp(2)$ with $f_B(h) < 0$.   As before, we let $h_2,h_4\in \mathbb{C}^4$ denote the first four entries of the $2$nd and $4$th columns of $h$.

To that end, we let $h_0\in Sp(2)\subseteq SU(5)$ denote any element with $$h_2 = \frac{\overline{s}_5}{|s_5|};$$ such an $h$ exists due to Proposition \ref{prop:emb}.

We claim that $f(h_0) < 0$.  We will prove this via a series of lemmas which estimate each of the terms $(|(Bh_0)_{\ell 2}|^2 + |(Bh_0)_{\ell 4}|^2)q_\ell$ of $f_B$.  The idea behind the choice of $h_0$ is to make the $\ell = 5$ term dominate this sum.  Since $q_5$ is negative by assumption, this gives $f(h_0)< 0$ as desired.

We begin with a bound on $|(Bh)_{\ell 2}|^2 + |(Bh)_{\ell 4}|^2$, which is valid for any $h$.

\begin{lem}\label{lem:q1bound} For any $h\in Sp(2)\subseteq SU(5)$, $$|(Bh)_{\ell 2}|^2 + |(Bh)_{\ell 4}|^2 \leq |s_{\ell}|^2.$$

\end{lem}

\begin{proof} As mentioned before, $(Bh)_{\ell 2} = s_{\ell}\cdot h_2$, and likewise for $(Bh)_{\ell 4}$.  The inequality we are trying to prove is invariant under scalings of the vector $s_{\ell}$, hence we may assume that $|s_{\ell}| = 1$.  Thus, we must show that $$|s_{\ell}\cdot h_2|^2 + |s_{\ell}\cdot h_4|^2\leq 1.$$

Because of the form of $h\in Sp(2)\subseteq SU(5)$, we see that the dot product $h_2\cdot \overline{h}_4 = 0$, so these vectors are orthonormal.  It follows that there is an element in $SU(4)$ which transforms $h_2$ to $(1,0,0,0)$, $h_4$ to $(0,1,0,0)$, and $s_{\ell} = (B_{\ell 1}, B_{\ell 2}, B_{\ell 3}, B_{\ell 4})$ to a unit vector of the form $(a,b,c,0)$

Then $$|(Bh)_{\ell 2}|^2 + |(Bh)_{\ell 4}|^2 = |s_{\ell}\cdot h_2|^2 + |s_{\ell}\cdot h_4|^2 = a^2 + b^2 = 1-c^2,$$ so is maximized when $c=0$, where the maximum value is $1$.

\end{proof}


We now find bounds for the lengths of the entries $\frac{\overline{B_{5j}}}{|s_5|}$ of $h_2$, when $j \in \{1,2,4\}$.

\begin{lem}\label{lem:ineq}  For any $B\in V$ and any $j\in \{1,2,4\}$,  $$\frac{|B_{5j}|}{|s_5|} \leq \frac{1}{14 \sqrt{q_m}}.$$

\end{lem}

\begin{proof}  Since $B\in V$, $s_5=|(B_{51}, B_{52}, B_{53}, B_{54})| > 7/8$ and $|B_{5j}| < \frac{1}{16 \sqrt{q_m}}$ for $j\in \{1,2,4\}$, we see $$\frac{|B_{5j}|}{|s_5|}  < \frac{1/(16\sqrt{q_m})}{7/8} = \frac{1}{14\sqrt{q}_m}.$$

\end{proof}

We can now bound the terms $(|(Bh_0)_{\ell 2}|^2 + |(Bh_0)_{\ell 4}|^2)q_\ell$ of $f_B$ when $\ell \in \{2,3,4\}$.

\begin{lem}\label{lem:234}

For $h_0$ as above and $\ell\in \{2,3,4\}$, we have $$(|(Bh_0)_{\ell 2}|^2 + |(Bh_0)_{\ell 4}|^2)q_\ell \leq 2/9.$$

\end{lem}

\begin{proof}

First observe that the triangle inequality gives \begin{align*}|(Bh_0)_{\ell 2}| &= |s_{\ell}\cdot h_2|\\ &\leq |B_{\ell 1}|\frac{| B_{51}|}{|s_5|} + |B_{\ell 2}|\frac{| B_{52}|}{|s_5|} + |B_{\ell 3}|\frac{| B_{53}|}{|s_5|} +|B_{\ell 4}|\frac{| B_{54}|}{|s_5|}.\end{align*}

Since $B\in V$, the terms involving $B_{ij}$ for $(i,j)\neq (2,2), (3,4)$ are bounded by $\frac{1}{16\sqrt{q_m}}$.  From Lemma \ref{lem:ineq}, $|B_{22}|\frac{|B_{52}|}{|s_5|}$ and $|B_{34}|\frac{|B_{54}|}{|s_5|}$ are both less than $\frac{1}{14\sqrt{q_m}}$.  Thus, for fixed $\ell\in \{2,3,4\}$, $|(Bh_0)_{\ell 2}| \leq \frac{3}{16 \sqrt{q_m}} + \frac{1}{14\sqrt{q_m}} < \frac{1}{3\sqrt{q_m}}$.

A similar argument applies to $|(Bh_0)_{\ell 4}|$.  Thus, $$(|(Bh)_{\ell 2}|^2 + |(Bh)_{\ell 4}|^2)q_\ell \leq \left(\frac{1}{9 q_m} + \frac{1}{9q_m}\right)q_\ell = \frac{2}{9} \frac{q_\ell}{q_m} \leq \frac{2}{9}.$$

\end{proof}

We are now ready to complete the proof that $f_B(h_0) < 0$.

\begin{proposition}  For $h_0$ as above, $f_B(h_0) < 0$.
\end{proposition}

\begin{proof} From Lemma \ref{lem:234}, the total contribution from the $q_2,q_3$ and $q_4$ terms of $f_B$ is bounded above by $3\cdot \frac{2}{9} = 2/3$.  Thus, if we can show that the $q_1$ and $q_5$ terms contribute less than $-2/3$, we will have shown $f_B(h_0) < 0$.

By the definition of $h_0$, $ |(Bh_0)_{52}|^2 = |s_5|^2$ and it is easy to verify that $(Bh_0)_{54} = 0$.  Since $B\in V$ and $q_5\leq -1$, we find $$\left( |(Bh_0)_{52}|^2 + |(Bh_0)_{54}|^2\right) q_5 < -7/8.$$  Of course, if $q_1<0$, it now follows easily that the contribution of the $q_1$ and $q_5$ terms is at most $-7/8 < -2/3$, completing the proof in this case.

Thus, we may assume $q_1 > 0$.  Then Lemma \ref{lem:q1bound} yields a bound of $|s_1|^2 q_1$ for the $q_1$ term of $f_B$.

We thus compute \begin{align*} \sum_{\ell \in \{1,5\}} \left( |(Bh)_{\ell 2}|^2 + |(Bh)_{\ell 4}|^2\right) q_\ell &\leq |s_1|^2 q_1 + |s_5|^2 q_5\\ &< |s_5|^2(q_1+ q_5)\\ &\leq -2|s_5|^2\\ & < -7/4. \end{align*}  In the above displayed inequalities, the second inequality uses the fact that $B\in V$ implies $|s_1| < |s_5|$, the third inequality is using the fact that $q_5 + q_1 < 0$ and that both are odd integers, and the last inequality comes from the fact that $B\in V$.

Since $-7/4 < -2/3$, it follows that $f(h_0) < 0$, as claimed.

\end{proof}

We may now complete the proof of Theorem \ref{thm:main}

\begin{proof}(Proof of Theorem \ref{thm:main}) Suppose $\overline{q} = (q_1,..., q_5)$ is admissible.  For the first part of Theorem \ref{thm:main}, note that, from Theorem \ref{thm:qp}, $\mathcal{B}_{\overline{q}}$ is quasi-positively curved, except when $\overline{q}$ is a permutation of $ (1,1,1,-1,-3)$.  In this exceptional case, Proposition \ref{prop:exzero} shows it has zero-curvature planes at every point.

We now prove the second part of Theorem \ref{thm:main}.  As Kerin \cite{Ke1} has shown $\mathcal{B}_{\overline{q}}$ with $\overline{q} = (1,1,1,1,-1)$ is almost positively curved, the backwards implication is clear.  Now, suppose $\mathcal{B}_{\overline{q}}$ is almost positively curved.  If some $q_i + q_j < 0$ for distinct $i,j,$ then Theorem \ref{thm:notalmost} shows that $\mathcal{B}_{\overline{q}}$ is not almost positively curved.  Thus, we must have $q_i + q_j \geq 0$ for all distinct $i,j$.

If all $q_i + q_j >0$, then it is well known $\mathcal{B}_{\overline{q}}$ has positive curvature \cite{EKS,Baz1}, so we may assume $q_i + q_j = 0$ for some distinct $i, j$.

Then, as shown by \cite[Lemma 3.4]{Ke1}, up to isometry, we must be in one of two cases: $(q_1,...,q_5) = (1,1,1,n,-n)$ or $(1,1,-3,n,-n)$ with $n\geq 1$ odd.  However, of these, only the first with $n=1$ has all $q_i + q_j \geq 0$.

\end{proof}

\bibliographystyle{plain}
\bibliography{bibliography.bib}

\end{document}